\documentclass[11pt]{article}       
\usepackage{amsmath,amssymb,graphicx,float,url,enumitem,xcolor,color,subcaption,mathtools}
\usepackage{placeins}
\usepackage{amsthm}
\usepackage{ragged2e}

\usepackage[colorlinks=true]{hyperref}

\usepackage{hyperref}
\hypersetup{hypertex=true,
            colorlinks=true,
            linkcolor=blue,
            anchorcolor=blue,
            citecolor=blue}

\newcommand*\myat{{\fontfamily{ptm}\selectfont @}}

\def \R{{\mathbb R}}
\def \P{{\mathcal P}}
\def \M{{\mathcal M}}
\def \N{{\mathbb N}}
\def \ds{\displaystyle}

\renewenvironment{proof}{\noindent{\bfseries Proof.}}{\qed} 
\newtheorem{theorem}{Theorem}[section]
\newtheorem{lemma}[theorem]{Lemma}

\newtheorem{corollary}[theorem]{Corollary}

\newtheorem{remark}[theorem]{Remark}

\numberwithin{equation}{section}

\begin{document}

\title{Group invariant solutions for the planar Schr\"{o}dinger-Poisson system
}


\author{Ganglong Zhou\footnote{This research did not receive any specific grant from funding agencies in the public, commercial, or
not-for-profit sectors.}\;\;\thanks{School of Mathematical Sciences, East China Normal University, Shanghai, PR China, 200241 (52195500021\myat stu.ecnu.edu.cn)}
}



\date{}
\providecommand{\keywords}[1]{\small{\textbf{Key words.} #1}}


\maketitle

\begin{abstract}

This paper is concerned with the following planar Schr\"{o}dinger-Poisson system
\begin{equation*}
\begin{cases}
  -\triangle{u}+V(x)u+\phi{(x)}|u|^{p-2}u=f(x,u),&\text{in $\mathbb{R}^{2}$},\\
  \triangle{\phi}=|u|^{p},&\text{in $\mathbb{R}^{2}$},
\end{cases}
\end{equation*}
where $p\geq2$ is a constant, $V(x)$ and $f(x,t)$ are continuous, mirror symmetric or rotationally periodic functions. By assuming that the nonlinearity $f(x,t)$ has critical exponential growth w.r.t.~$H^1(\R^2)$, we obtain a nontrivial solution or a ground state solution of Nehari type to the above system. Our results extend previous works of Cao, Dai and Zhang \cite{Cao_Dai_Zhang}, Chen and Tang \cite{Chen_Tang2020-2}. We handle more general nonlinearities $f$ with weaken constraint at infinity, and we assume only the (AR) type condition to take place of the monotonicity assumption (Theorem \ref{thm1}). We considered all the cases $p\geq2$, and we show the existence of solutions with multiple types of symmetry. As in \cite{Chen_Tang2020-2}, we adopt a version of mountain pass structure which provides a Cerami sequence (Lemma \ref{lem9}), with two innovative points. First, we make a key observation for the sign of a crucial part of the energy functional corresponding to the nonlocal term $\phi|u|^{p-2}u$ (Lemma \ref{lem6}), and secondly we adopt a new Moser type functions (Lemma \ref{lem10}) to ensure the boundedness and compactness of the Cerami sequence. Moreover, our approach works also for the subcritical growth case (Theorem \ref{R15}), and generalizes recent works \cite{Liu_Radulescu_Tang_Zhang,Cao_Dai_Zhang,Chen_Tang2020-2}.
\end{abstract}


\section{Introduction}
\label{sec:intro}

The present paper is concerned with the existence of solution to the planar Schr\"{o}dinger-Poisson system
\begin{equation*}
\begin{cases}
  -\triangle{u}+V(x)u+\phi{(x)}|u|^{p-2}u=f(x,u),&\text{in $\mathbb{R}^{2}$},\\
  \triangle{\phi}=|u|^{p},&\text{in $\mathbb{R}^{2}$}.
\end{cases}
\leqno{(P)}
\end{equation*}
where $p\geq2$, $V,f$ are continuous, mirror symmetric or rotationally periodic functions; and $f(x,t)$ has exponential critical growth in the Trudinger-Moser sense (see \cite{Cao}).

In last decades, considerable attentions have been paid for the following Schr\"{o}dinger-Poisson system:
\begin{equation*}
\begin{cases}
  -\triangle{u}+V(x)u+K(x)\phi{(x)}|u|^{p-2}u=f(x,u),&\text{in $\mathbb{R}^N$},\\
  \triangle{\phi}=K(x)|u|^{p},&\text{in $\mathbb{R}^N$}.
\end{cases}
\leqno{(SP)}
\end{equation*}
with various conditions on the parameters $p,N$ and functions $V,K,f$. In three-dimensional case $N = 3$, this kind of system arises in many contexts of physics, for example, in quantum mechanics \cite{Benguria_Brezis_Lieb1981,Catto_Lions1993,Lieb1981} and semiconductor theory \cite{Benci_Fortunato1998,Benci_Fortunato2002,Lions1984,Markowich_Ringhofer_Schmeiser1984}. In \cite{Benci_Fortunato1998}, the equation $(SP)$ was introduced as a model describing solitary waves, for nonlinear stationary
equations of Schr\"{o}dinger type interacting with an electrostatic field, where the unknown functions $u$ and $\phi$ denote wave function for particles and potential respectively. Mathematically speaking, for each $u\in H^1(\mathbb{R}^N)$, the second equation in $(SP)$ determines a unique solution $\phi_u$ in $D^{1,2}(\mathbb{R}^N)$, i.e.~$\phi_u:=\Gamma_N\ast|u|^p$, where
\begin{equation*}
  \Gamma_N=
  \begin{cases}
    \frac{1}{2\pi}\ln|x|,&\text{$N=2$},\\
    \frac{1}{N(2-N)\omega_N}|x|^{2-N},&\text{$N\geq3$},
  \end{cases}
\end{equation*}
and $\omega_N$ is the volume of the unit ball in $\mathbb{R}^N$. When $N\geq3$, many minimization techniques were used to the system $(SP)$. We refer the readers to \cite{Ambrosetti_Ruiz2008,Cerami_Vaira2010,Ruiz2006,Zhao_Zhao2008} and references therein.

In the following, let us focus on the two dimensional case. In 2008, Stubbe \cite{Stubbe} considered the equation
\begin{equation*}
  -\triangle{u}+\lambda u+\phi{(x)}u=0,\;\; \triangle{\phi}=u^2,\quad \text{in $\mathbb{R}^{2}$},
\leqno{(S1)}
\end{equation*}
where $\lambda\in\mathbb{R}$ is a constant. He set up a variational framework for $(S1)$ with a subspace $Z$ of $H^1(\mathbb{R}^2)$:
$$Z:=\left\{u\in H^1(\mathbb{R}^2):\int_{\mathbb{R}^2}\ln(1+|x|)u^2dx<\infty\right\}.$$
He proved that there exists a unique radial ground state solution for any $\lambda\geq0$. In addition, he proved
that there exists a negative number $\lambda^{\ast}$ such that for any $\lambda\in(\lambda^{\ast},0)$ there are two radial ground states with different $L^2$ norms.

In 2016, Cigolani and Weth \cite{Cingolani_Weth} considered the equation $(P)$ with $p = 2$ and $f(x, u) = b|u|^{q-2}u$, that is
\begin{equation*}
\begin{cases}
  -\triangle{u}+V(x)u+\phi{(x)}u=b|u|^{q-2}u,&\text{in $\mathbb{R}^{2}$},\\
  \triangle{\phi}=u^2,&\text{in $\mathbb{R}^{2}$},
\end{cases}
\leqno{(S2)}
\end{equation*}
where $b\geq0,q\geq4$ are constants and $V\in C(\mathbb{R}^2,(0,\infty))$ is $\mathbb{Z}^2$ periodic.
Using the concentration-compactness theory, they proved that the equation $(S2)$ has a ground state $u\in X_2$ and a solution sequence $\{u_n\}_n\subset X_2$ such that $\lim_{n\rightarrow\infty}J(u_n)=\infty$. Here
$$X_2:=\left\{u\in H^1(\mathbb{R}^2):\int_{\mathbb{R}^2}\Big[|\nabla u|^2+V(x)u^2+\ln(1+|x|)u^2\Big]dx<\infty\right\}$$
and $J$ is the energy functional associated to $(S2)$.

 In 2021, Cao et al. \cite{Cao_Dai_Zhang} generalized the above consideration for $p \geq 2$,
\begin{equation*}
\begin{cases}
  -\triangle{u}+V(x)u+ \phi{(x)}|u|^{p-2}u=b|u|^{q-2}u,&\text{in $\mathbb{R}^{2}$},\\
  \triangle{\phi}=|u|^p,&\text{in $\mathbb{R}^{2}$},
\end{cases}
\leqno{(S3)}
\end{equation*}
where $b\geq0,p\geq2,q\geq2p$ are constants and $V\in C(\mathbb{R}^2,(0,\infty))$ is $\mathbb{Z}^2$ periodic. With similar method in \cite{Cingolani_Weth}, they obtained the existence of a positive ground state solution of $(S3)$ in $X_p$, where
  $$X_p:=\left\{u\in H^1(\mathbb{R}^2):\int_{\mathbb{R}^2}\Big[|\nabla u|^2+V(x)u^2+\ln(1+|x|)|u|^p\Big]dx<\infty\right\}.$$

In 2020, Chen and Tang \cite{Chen_Tang2020-2} considered the equation $(P)$ with $p = 2$, i.e.
\begin{equation*}
\begin{cases}
  -\triangle{u}+V(x)u+\phi{(x)}u=f(x,u),&\text{in $\mathbb{R}^{2}$},\\
  \triangle{\phi}=u^2,&\text{in $\mathbb{R}^{2}$},
\end{cases}
\leqno{(S4)}
\end{equation*}
where $V\in C(\mathbb{R}^2,[0,\infty))$ is axially symmetry and $f\in C(\mathbb{R}^3)$ is of subcritical or critical exponential growth in the sense of Trudinger-Moser. More precisely, we say that $f(x,t)$ has subcritical exponential growth at $t=\pm\infty$ if it verifies
\begin{itemize}
\item[(F1')] (subcritical case) $\sup_{x\in\mathbb{R}^2,|s|\leq A}|f(x,s)|<+\infty$ for every $A > 0$ and
\begin{equation*}
\lim_{|t|\rightarrow\infty}\frac{|f(x,t)|}{e^{\alpha t^2}}=0,\;\mbox{uniformly in}\;\mathbb{R}^2\;\mbox{for all}\;\alpha>0;
\end{equation*}
\end{itemize}
and $f(x,t)$ is said having the critical exponential growth at $t=\pm\infty$ if it verifies
\begin{itemize}
\item[(F1)] (critical case) $\sup_{x\in\mathbb{R}^2,|s|\leq A}|f(x,s)|<+\infty$ for every $A > 0,$ and there exists $\alpha_0>0$ such that
\begin{equation*}
\lim_{|t|\rightarrow\infty}\frac{|f(x,t)|}{e^{\alpha t^2}}=0,\;\;\mbox{uniformly in}\;\mathbb{R}^2\;\mbox{for all}\;\alpha>\alpha_0
\end{equation*}
while
\begin{equation*}
\lim_{|t|\rightarrow\infty}\frac{|f(x,t)|}{e^{\alpha t^2}}=+\infty,\;\mbox{uniformly in}\;\mathbb{R}^2\;\mbox{for all}\;\alpha<\alpha_0.
\end{equation*}
\end{itemize}
This notion of criticality was introduced by Adimurthi and Yadava \cite{Adimurthi_Yadava}.

For the critical growth case, Chen and Tang \cite{Chen_Tang2020-2} established the existence of a ground state solution for $(S4)$ by assuming a series of technical conditions on $V$ and $f$:
\begin{itemize}
\item[(V0)]  $V\in C(\mathbb{R},[0,\infty))$ and $\liminf_{|x|\rightarrow\infty} V(x)>0$;
\item[(CF1)] $V(x):=V(x_1,x_2)=V(|x_1|,|x_2|)$ for all $x\in\mathbb{R}^2$, $f(x,t):=f(x_1,x_2,t)=f(|x_1|,|x_2|,t)$ for all $(x,t)\in\mathbb{R}^2\times\mathbb{R}$;
\item[(CF2)] $f(x,t)t>0$ for all $(x,t)\in\mathbb{R}^2\times\mathbb{R}\backslash\{0\}$, and there exists $M_0>0$ and $t_0>0$ such that
$$F(x,t)\leq M_0|f(x,t)|,\;\;\forall\; x\in\mathbb{R}^2,|t|\geq t_0,$$
where $F(x,t):= \ds\int_0^tf(x,s)ds;$
\item[(CF3)] There exists $\kappa>0$, $\liminf_{|t|\rightarrow\infty}\frac{t^2F(x,t)}{e^{\alpha_0 t^2}}\geq\kappa>\frac{2}{\alpha_0^2\rho^2}$ uniformly on $x\in\mathbb{R}^2$, where $\rho\in(0,1/2)$ satisfies $\rho^2\max_{|x|\leq\rho}V(x)\leq1$;
\item[(CF4)] $\frac{f(x,t)-V(x)t}{|t|^{3}}$ is non-decreasing on $t\in(-\infty,0)$ and $t\in(0,\infty)$.
\end{itemize}

Here we will prove the existence of a nontrivial solution to the equation $(P)$, not only for all $p \geq 2$, but also for much more general nonlinearities $f$ and more general potentials $V$ (see Remark \ref{newR} below).

To describe our main results, we introduce the following notations: Let us view $\mathbb{R}^2$ as $\mathbb{C}$, Let $k \in \N, k\ge 2$, we say that $v \in {\mathcal P}_k$ if $v(ze^{2\pi i/k})=v(z)$ over $\mathbb{C}$. We define
$$E_{k,p}:= X_p\cap \P_k,\quad \mathcal{V}_{k,1}:= C(\mathbb{C})\cap \P_k$$
$$\mathcal{F}_{k,1}:= \{f\in C(\mathbb{C}\times\mathbb{R}):f(\cdot,t) \in \P_k, \forall\;t\in\mathbb{R}\}.$$
We say that $v$ is mirror symmetric, denoted by $v \in \M$ if $v(\overline z)= v(z)$ in ${\mathbb C}$. Let
$$T_{k,p} := E_{k,p}\cap \M, \quad \mathcal{V}_{k,2}:= \mathcal{V}_{k,1} \cap \M, $$
$$\mathcal{F}_{k,2}:=\{f\in\mathcal{F}_{k,1}: f(\cdot,t) \in \M, \forall\; t\in\mathbb{R}\}.$$

Our main result is stated as follows.
\begin{theorem}
\label{thm1}
Let $p\geq2$, V and f satisfy (V0) and (F1). Assume that
\begin{itemize}
\item[(VF)] $V\in\mathcal{V}_{k,1}\;\mbox{and}\;f\in\mathcal{F}_{k,1}$ with $k\geq4$ or $V\in\mathcal{V}_{k,2}\;\mbox{and}\;f\in\mathcal{F}_{k,2}$ with $k\geq2$;
\item[(F2)] $f(x,t)t \ge 0$ for all $(x,t)\in\mathbb{R}^2\times\mathbb{R}\backslash\{0\}$, and there exists $M_0>0$ and $t_0>0$ such that $F(x,t)\leq M_0|f(x,t)|$ for $x\in\mathbb{R}^2$, $|t|\geq t_0;$
\item[(F3)] There exists $q\in\mathbb{R}$ such that $\liminf_{|t|\rightarrow\infty}\frac{|t|^qF(x,t)}{e^{\alpha_0 t^2}}=+\infty;$
\item[(F4)] (AR) type condition: There exist $\mu_1 > 2$, $\mu_2 <\frac{\mu_1}{2}-1$ such that
$$f(x,t)t\geq\mu_1F(x,t)-\mu_2V(x)t^2,\;\;\forall\;(x,t)\in\mathbb{R}^3.$$
\item[(F5)] If $p = 2$, $f(t) = o(t)$ as $t\rightarrow0$ uniformly on $\mathbb{R}^2;$ and if $p > 2$, $f(t)= O(t^{s_0})$ with $s_0 > 1$ as $t\rightarrow0$ uniformly on $\mathbb{R}^2.$
\end{itemize}
\noindent Then $(P)$ has a nontrivial solution $\bar{u}$. Moreover, $\bar{u}\in E_{k,p}$ if\; $V\in\mathcal{V}_{k,1}$ and $f\in\mathcal{F}_{k,1}$; $\bar{u}\in T_{k,p}$ if $V\in\mathcal{V}_{k,2}$ and $f\in\mathcal{F}_{k,2}$.
\end{theorem}

\begin{remark}
\label{newR}
  Comparing to Chen and Tang \cite[Theorem 4]{Chen_Tang2020-2}, we have generalized or weaken all the assumptions (CF1)-(CF4), to (VF) and (F2)-(F4) respectively. More precisely,
  \begin{itemize}
  \item (CF1) means just $V\in\mathcal{V}_{2,2}$ and $f\in\mathcal{F}_{2,2}$, hence is a special case of (VF);
  \item The condition (CF4) is a significant improvement of (F4) since the (AR) type condition does not require any monotonicity assumption;
  \item The condition (F3) is also less restrictive than (CF3) for the behavior of $f$ at infinity;
  \item (F2) improves slightly (CF2) where $f(x,t)t>0$ is replaced by $f(x,t)t\geq0$;
  \end{itemize}
 Finally, \cite{Chen_Tang2020-2} used implicitly the condition $f(t) = o(t)$ with $p =2$ in (F5), seeing the proof of Lemma 2.6 there.
\end{remark}

With a similar condition to (CF4), we can obtain the least energy solution on the corresponding Nehari manifold.
\begin{theorem}
\label{thm3}
Let $p\geq2$, $V$ and $f$ satisfy (V0), (VF), (F1)-(F3) and (F5). Furthermore, we assume that
\begin{itemize}
\item[(F4')] For any $p\geq2$, $g_p(x,t)$ is non-decreasing on $t\in(-\infty,0)$ and $t\in(0,\infty)$ respectively, where
\begin{equation*}
g_p(x,t):=
  \begin{cases}
    \frac{f(x,t)-V(x)t}{|t|^{2p-1}}, & p=2 ,\\
    \frac{f(x,t)-\mu V(x)t}{|t|^{2p-1}}, & p>2, \mu < 1.
  \end{cases}
\end{equation*}
\end{itemize}
Then $(P)$ has a nontrivial solution $\bar{u}$. Moreover, if $V\in\mathcal{V}_{k,1}$ and $f\in\mathcal{F}_{k,1}$, then $\bar{u}\in E_{k,p}$ satisfies $$\Phi(\bar{u}) = \min_{\mathcal{N}_1}\Phi;$$
if $V\in\mathcal{V}_{k,2}$ and $f\in\mathcal{F}_{k,2}$, then $\bar{u}\in T_{k,p}$ satisifies $$\Phi(\bar{u}) = \min_{\mathcal{N}_2}\Phi.$$ Here the symbols $\Phi, \mathcal{N}_1,\mathcal{N}_2$ are given by \eqref{Phi}, \eqref{Ner} in the next section.
\end{theorem}


Comparing to the approach in \cite[Theorem 4]{Chen_Tang2020-2}, we still adopt the consideration of Cerami sequence, but we make a key observation which is actually Lemma \ref{lem6} on the functional $I_0$. Furthermore, we construct also a new Moser type function to weaken the condition (CF3) to (F3). Clearly, our approach works also for subcritical case, which generalizes Theorem 1.2 in \cite{Chen_Tang2020-2}, see Theorem \ref{R15} in section 6.

This paper is organized as follows: In Section 2, we present some basic results, in particular we show that the energy functional corresponding to the nonlocal term $\phi|u|^{p-2}u$ is non positive, see Lemmas \ref{lem6}. In secction 3, we establish the mountain pass structure, using a new test function, see Lemma \ref{lem10} below. In sections 4-5, we give the proof of Theorems \ref{thm1} and \ref{thm3} respectively. In section 6, we discuss some open problems about 2-D Schr\"{o}dinger-Poisson equation, and state our existence result for the subcritical growth case.

\section{Preliminaries}
\label{sec:prelim}
In this section, we will give some preliminary definitions and basic facts. In the following, the letter $C$ denotes generic positive
constants and $\|\cdot\|_q$ denotes the standard norm in $L^q({\mathbb{R}^2})$.

The function space $X_p$ is a Banach space equipped with the norm
\begin{equation}
  \|u\|_{X_p}:=\|u\|+\|u\|_{\ast},
\end{equation}
where
\begin{equation}
  \|u\|_{\ast}:=\left(\int_{\mathbb{R}^2}\ln(1+|x|)|u|^{p}dx\right)^{\frac{1}{p}};
\end{equation}
while
\begin{equation}
\label{3}
  \|u\|:=\left(\int_{\mathbb{R}^2}\Big[|\nabla u|^2+V(x)u^{2}\Big]dx\right)^{\frac{1}{2}}
\end{equation}
is induced by the scalar product
\begin{equation}
  \langle u,v\rangle:=\int_{\mathbb{R}^2}(\nabla u\cdot\nabla v+V(x)uv)dx.
\end{equation}

We will use the following bilinear functionals introduced by Stubbe \cite{Stubbe}:
\begin{equation}
\label{5}
  A_{1}(u,v):=\int_{\mathbb{R}^2}\int_{\mathbb{R}^2}\ln(1+|x-y|)u(x)v(y)dxdy;
\end{equation}
\begin{equation}
\label{6}
  A_{2}(u,v):=\int_{\mathbb{R}^2}\int_{\mathbb{R}^2}\ln\left(1+\frac{1}{|x-y|}\right)u(x)v(y)dxdy;
\end{equation}
\begin{equation}
\label{7}
  A_{0}(u,v):=A_{1}(u,v)-A_{2}(u,v)=\int_{\mathbb{R}^2}\int_{\mathbb{R}^2}\ln(|x-y|)u(x)v(y)dxdy.
\end{equation}
By Hardy-Littlewood-Sobolev inequality (see \cite{Lieb1983}), there exists $C>0$ such that for any $u,v\in L^{4/3}(\mathbb{R}^2)$,
\begin{equation}
\label{HLS}
  |A_2(u,v)|\leq\int_{\mathbb{R}^2}\int_{\mathbb{R}^2}\frac{1}{|x-y|}u(x)v(y)dxdy\leq C\|u\|_{\frac{4}{3}}\|v\|_{\frac{4}{3}}.
\end{equation}
Corresponding to \eqref{5}-\eqref{7}, we define
\begin{equation}
\label{9}
  I_i(u):=A_{i}(|u|^{p},|u|^{p}),\quad i=0,1,2.
\end{equation}
The following bound for $I_2(u)$ is a direct consequence of \eqref{HLS}:
\begin{equation}
\label{10}
  |I_2(u)|\leq C\|u\|_{\frac{4p}{3}}^{2p},\;\;\forall\; u\in L^{\frac{4p}{3}}(\mathbb{R}^2),\;\forall\;p\geq 1.
\end{equation}
We can write the associated functional of the equation $(P)$ in the following form
\begin{equation}
\label{Phi}
  \Phi(u)=\frac{1}{2}\|u\|^2+\frac{1}{4p\pi}I_0(u)-\int_{\mathbb{R}^2}F(x,u)dx
\end{equation}
and the associated Nehari manifold of the functional \eqref{Phi} are
\begin{align}
\label{Ner}
\begin{split}
\mathcal{N}_1:=\{u\in E_{k,p}\backslash\{0\}:\langle \Phi'(u),u\rangle=0\},\\
\mathcal{N}_2:=\{u\in T_{k,p}\backslash\{0\}:\langle \Phi'(u),u\rangle=0\}.
\end{split}
\end{align}

Next, we state several useful lemmas.
\begin{lemma}
\label{M-T}
Let $u\in H^{1}(\mathbb{R}^2)$, then for any $\alpha > 0$,
  $$\int_{\mathbb{R}^2}\left(e^{\alpha u^2}-1\right)dx<\infty.$$
Furthermore, given $M>0$, $\alpha\in(0,4\pi)$, there exists a constant $C(M,\alpha)$ such that for all $u\in H^{1}(\mathbb{R}^2)$ satisfying $\|\nabla u\|_{2}\leq 1$, $\|u\|_2\leq M$, there holds
  $$\int_{\mathbb{R}^2}\left(e^{\alpha u^2}-1\right)dx<C(M,\alpha),$$
\end{lemma}
The above Lemma was first established by Cao in \cite{Cao} (see also \cite{Adachi_Tanaka,Cassani_Sani_Tarsi}).

\begin{lemma}
\label{lem5}
  Assume that (V0),(F1),(F5) hold. Then $I_i,\Phi\in C^1(X_p,\mathbb{R})$ and
\begin{align}
\label{diff}
\begin{split}
\langle I^{\prime}_i(u),v\rangle=2pA_i(|u|^p,|u|^{p-2}uv),\quad i=1,2\\
  \langle \Phi^{\prime}(u),v\rangle=\langle u,v\rangle+\frac{1}{2\pi}A_{0}(|u|^p,|u|^{p-2}uv)-\int_{\mathbb{R}^2}f(x,u)vdx.
\end{split}
\end{align}
\end{lemma}
For the sake of completeness, we present a proof of Lemma \ref{lem5} in the appendix. The following lemma is our first key observation.
\begin{lemma}
\label{lem6}
  $I_{0}(u)\leq0,\;\;\forall\; u\in X_p.$
\end{lemma}

  \begin{proof}
First, let $u\in C_0^\infty(\R^2)$ with ${\rm supp}(u) \subset B_{\frac{1}{2}}(0)$. Let $\phi_u$ be the Newton potential of $|u|^p$, i.e.~the unique solution of $\Delta\phi = |u|^p$, $\phi \in D^{1, 2}(\R^2)$. Then
    $$2\pi \int_{\mathbb{R}^2}\phi_u\triangle\phi_u dx= I_0(u)= \int_{B_{\frac{1}{2}}(0)}\int_{B_{\frac{1}{2}}(0)}|u(x)|^p|u(y)|^p\ln(|x-y|)dxdy\leq0.$$
Consider now $u\in{C}_{0}^{\infty}(\mathbb{R}^2,\mathbb{R})$, take $R>0$ such that $\mathrm{supp}(u)\subset B_R(0)$. Let $w(x)=u(2Rx)$, so ${\rm supp}(w) \subset B_{\frac{1}{2}}(0)$ and $\phi_w(x):=\frac{1}{4R^2}\phi_u(2Rx).$ Hence
    $$\frac{1}{2\pi}I_0(u)=\int_{\mathbb{R}^2}\phi_u\triangle\phi_u dx = 16R^4\int_{\mathbb{R}^2} \phi_w\triangle \phi_wdx\leq0.$$
    Now we conclude by density argument. Indeed, using standard argument, by cut-off and mollification, we can check readily that $C_0^\infty(\mathbb{R}^2)$ is dense in $X_p$, that is for any $u \in X_p$, there exist $\{u_n\}_n\subset C_0^\infty(\mathbb{R}^2)$ such that $\lim_{n\rightarrow\infty}\|u_n-u\|_{X_p}=0.$ By the fact $I_0=I_1-I_2$ and Lemma \ref{lem5}, we conclude that $I_{0}(u) \leq0.$
  \end{proof}

\begin{corollary}
\label{cor7}
  Assume that (V0) and (F1) hold. Then
  $$\lim_{t\rightarrow\infty}\Phi(t\omega)=-\infty,\;\;\forall\; \omega\in X_p\backslash\{0\}.$$
\end{corollary}
\begin{proof}
For any $\omega\in X_p\backslash\{0\}$, there exists $\delta>0$ such that $m\{|\omega(x)|\geq\delta\}>0$. By Lemma \ref{lem6} and (F1), one has
  \begin{align*}
    \Phi(t\omega)&=\frac{t^2}{2}\|\omega\|^2+\frac{t^4}{4p\pi}I_0(\omega)-\int_{\mathbb{R}^2}F(x,t\omega)dx\\
    &\leq\frac{t^2}{2}\|\omega\|^2-C\int_{\{|\omega(x)|\geq\delta\}}e^{\alpha_0\delta^2t^2/2}dx \longrightarrow -\infty
  \end{align*}
as $|t|\to\infty$.
\end{proof}

\medskip
The following lemma is inspired by \cite{{Chen_Tang2020-2}}.
\begin{lemma}
\label{lem8}
  Assume that (V0) and (VF) hold. Then there exists $C_k>0$ such that
  \begin{equation}
  \label{14}
    A_{1}(|u|^p,|v|^p)\geq C_k\|u\|_{\ast}^p\|v\|^p_p,\;\;\forall\;u,v\in E_{k,p} \mbox{  or  } \;u,v\in T_{k,p}.
  \end{equation}
\end{lemma}
  \begin{proof}
We only consider the case $E_{k,p}$ since $T_{k,p} \subset E_{k,p}$. Let $\Omega_1 :=\{(x_1,x_2)\in\mathbb{R}^2:x_i\geq0\}$, $\Omega_2 = -\Omega_1$.
      For any $(x,y)\in \Omega_1\times \Omega_2$, one has
      $$|x-y|^2=|x|^2+|y|^2-2x\cdot y\geq|x|^2+|y|^2.$$
      Then it follows from the definition of $E_{k,p}$ and $k\geq4$ that
      \begin{align*}
        A_1(|u|^p,|v|^p)&=\int_{\mathbb{R}^2}\int_{\mathbb{R}^2}\ln(1+|x-y|)|u(x)|^p|v(y)|^pdxdy\\ &\geq\int_{\Omega_2}|v(y)|^{p}dy\int_{\Omega_1}\ln(1+|x-y|)|u(x)|^pdx\\
        &\geq\int_{\Omega_2}|v(y)|^{p}dy\int_{\Omega_1}\ln(1+|x|)|u(x)|^pdx\\
        &\geq\frac{1}{k^2}\int_{\mathbb{R}^2}|v(y)|^{p}dy\int_{\mathbb{R}^2}\ln(1+|x|)|u(x)|^{p}dx\\
        &\geq C_k\|u\|^{p}_{\ast}\|v\|^{p}_{p},\;\;\forall\;u,v\in E_{k,p}.
      \end{align*}
So we are done. \end{proof}

\section{Variational framework}
In this section, we will quote a version of Mountain Pass Theorem and prepare the proof of Theorem \ref{thm1}.
\begin{lemma}
\label{lem9}
  Let $Y$ be a real Banach space and $I\in C^1(Y,\mathbb{R})$. Let $S$ be a closed subset of $Y$ which disconnects $Y$ into distinct connected $Y_1$ and $Y_2$. Suppose further that $I(0)=0$ and\\
  (i) $0\in Y_1$ and there exists $\alpha>0$ such that $I|_{S}\geq\alpha,$\\
  (ii) There is $e\in Y_2$ such that $I(e)\leq0.$\\
  Then $I$ possess a $(Ce)_c$ sequence with $c\geq\alpha>0$ given by
  $$c=\inf_{\gamma\in\Gamma}\max_{t\in[0,1]}I(\gamma(t)),$$
  where
  $$\Gamma=\{\gamma\in C([0,1],X):\gamma(0)=0,\gamma(1)=e\}$$
  and $(Ce)_c$ sequence means a Cerami sequence $\{u_n\}\subset X$ such that $$I(u_n)\rightarrow{c},\;\|I'(u_n)\|_{Y^{\ast}}(1+\|u_n\|_{Y})\rightarrow0.$$
\end{lemma}
The proof of the above Lemma can be found in \cite[Theorem 3]{Silva_Vieria}. We state another key result which serves as a bridge between the mountain pass structure (see Lemma \ref{lem11}) and Theorem \ref{thm1}.

\begin{lemma}
\label{lem10}
  Assume that (V0),(F1) and (F3)-(F5) hold. Then there exists $n_0\in\mathbb{N}$ such that
  \begin{equation}
  \label{15}
    \max_{t\geq0}\Phi(t\omega_{n_0})<\frac{2\pi}{\alpha_0},
  \end{equation}
  where
  \begin{equation*}
    \omega_n(x)=
    \begin{cases}
      \frac{\sqrt{\ln n}}{\sqrt{2\pi}}-\frac{q\ln{(\ln{n})}}{2\sqrt{2\pi\ln{n}}},&\text{$0\leq|x|\leq(\ln{n})^{q/2}/n$};\\
      \frac{\ln(1/|x|)}{\sqrt{2\pi\ln n}},&\text{$(\ln{n})^{q/2}/n\leq|x|\leq1$};\\
      0,&\text{$|x|\geq1$}.
    \end{cases}
  \end{equation*}
  \end{lemma}

  \begin{proof}
Without loss of generality, we can fix $q\geq2$. Direct computation yields
    \begin{align}
    \label{16}
    \begin{split}
        \|\omega_n\|^2\leq&\int_{B_1}|\nabla\omega_n|^2dx+V_1\int_{B_1}\omega_n^2dx\\
        =&1-\frac{q\ln{(\ln{n})}}{2\ln{n}}+V_1\delta_n,
        \end{split}
    \end{align}
    where $V_1=\max_{x\in B_1}V(x)$ and $\delta_n = O(\frac{1}{\ln{n}})$ as $n \to \infty$.
    By $(F3)$, there exists $t_0>0$ such that
    \begin{equation}
    \label{17}
      \frac{|t|^qF(x,t)}{e^{\alpha_0t^2}}\geq1,\;\forall\;|t|\geq t_0.
    \end{equation}
    There are three cases for the value of $t$.

    \smallskip
    Case (i): $0 \leq t \le \sqrt{\frac{3\pi}{\alpha_0}}$. Then it follows from \eqref{16} and Lemma \ref{lem6} that
    \begin{align}
      \Phi(t\omega_n)&=\frac{t^2}{2}\|\omega_n\|^2+\frac{t^{2p}}{2p}I_0(\omega_n)-\int_{\mathbb{R}^2}F(x,t\omega_n)dx \leq\frac{1+V_1\delta_n}{2}t^2\leq\frac{7\pi}{4\alpha_0}
    \end{align}
    for large $n$.

    \smallskip
    Case (ii): $\sqrt{\frac{3\pi}{\alpha_0}} \le t\le \sqrt{\frac{8\pi}{\alpha_0}}$. Then $t\omega_n(x)\geq t_0$ for $x\in B_{(\ln{n})^{q/2}/n}$ and large $n$. Then it follows from \eqref{16},\eqref{17} and Lemma \ref{lem6} that
    \begin{align}
    \label{19}
    \begin{split}
      \Phi(t\omega_n)&=\frac{t^2}{2}\|\omega_n\|^2+\frac{t^{2p}}{2p}I_0(\omega_n)-\int_{\mathbb{R}^2}F(x,t\omega_n)dx\\
      &\leq\frac{1+V_1\delta_n}{2}t^2-\frac{q\ln(\ln{n})}{4\ln{n}}t^2-\frac{2^{q/2}\pi^{1+q/2}(\ln{n})^q}{n^2t^qT_n^{q/2}}e^{\frac{\alpha_0}{2\pi}t^2T_n}\\
      &\leq\frac{1+V_1\delta_n}{2}t^2-\frac{q\ln(\ln{n})}{4\ln{n}}t^2-\frac{\alpha_0^{q/2}\pi(\ln{n})^q}{2^{q}n^2T_n^{q/2}}e^{\frac{\alpha_0}{2\pi}t^2T_n}=:\varphi_n(t),
      \end{split}
    \end{align}
    where
    $$T_n:=\ln{n}-q\ln(\ln{n})+\frac{q^2\ln^2(\ln{n})}{4\ln{n}}.$$
    Let $t_n>0$ be the unique maximum of $\varphi_n$ in $\R_+$, then (as $n\to\infty$)
    \begin{align}
      t_n^2&=\frac{4\pi}{\alpha_0}\left[1+\frac{(q-1)\ln(\ln{n})}{2\ln{n}}+O\left(\frac{1}{\ln{n}}\right)\right]
    \end{align}
    and
    \begin{align}
    \label{21}
      \varphi_n(t)\leq\varphi_n(t_n)&=\frac{1+V_1\delta_n}{2}t_n^2-\frac{q\ln(\ln{n})}{4\ln{n}}t_n^2 + O\left(\frac{1}{\ln{n}}\right).
    \end{align}
    Combining \eqref{19}-\eqref{21}, one has
    \begin{align}
      \Phi(t\omega_n)&\leq\varphi_n(t_n) =\frac{2\pi}{\alpha_0}-\frac{\ln(\ln{n})}{2\ln{n}}+O\left(\frac{1}{\ln{n}}\right).
    \end{align}

    Case (iii): $t\ge \sqrt{\frac{8\pi}{\alpha_0}}$. Again $t\omega_n(x)\geq t_0$ for $x\in B_{(\ln n)^{q/2}/n}$ for large $n$, it follows from \eqref{16},\eqref{17} and Lemma \ref{lem6} that
    \begin{align}
    \begin{split}
      \Phi(t\omega_n)&=\frac{t^2}{2}\|\omega_n\|^2+\frac{t^{2p}}{2p}I_0(\omega_n)-\int_{\mathbb{R}^2}F(x,t\omega_n)dx\\
      &\leq\frac{1+V_1\delta_n}{2}t^2-\frac{2^{q/2}\pi^{1+q/2}(\ln{n})^q}{n^2t^qT_n^{q/2}}e^{\frac{\alpha_0}{2\pi}t^2T_n}\\
      &\leq\frac{1+V_1\delta_n}{2}t^2-\frac{2^{q/2}\pi^{1+q/2}(\ln{n})^q}{t^qT_n^{q/2}}\exp\left[2\left(\frac{\alpha_0}{4\pi}t^2-1\right)T_n\right]\\
      &\leq\frac{4\pi(1+V_1\delta_n)}{\alpha_0}-\frac{\alpha_0^{q/2}\pi(\ln{n})^{q/2}}{2^{q}}n^{2} \\
      &\leq 0
      \end{split}
    \end{align}
    for large $n$. To get the third inequality, we used the fact that the function
    $$\frac{1+V_1\delta_n}{2}t^2-\frac{2^{q/2}\pi^{1+q/2}(\ln{n})^q}{t^qT_n^{q/2}}\exp\left[2\left(\frac{\alpha_0}{4\pi}t^2-1\right)T_n\right]$$
    is decreasing on $t\ge \sqrt{\frac{8\pi}{\alpha_0}}$ when $n$ is large enough.
  Combining the conclusions for cases (i)-(iii), the proof is completed.\end{proof}

\smallskip
Now we show the existence of Cerami sequence.
\begin{lemma}
\label{lem11}
  Assume that (V0),(VF),(F1) and (F5) hold. Then there exists a constant $\tilde{c}\in(0,\sup_{t\geq0}\Phi(t\omega_{n_0})]$ and a Cerami sequence $\{u_n\}\subset E_{k,p}$ such that
  \begin{align}
  \label{24}
    \Phi(u_n)\rightarrow\tilde{c},\;\|\Phi^{\prime}(u_n)\|_{X_p^{\ast}}(1+\|u_n\|_{X_p})\rightarrow0.
  \end{align}
  \end{lemma}
  \begin{proof}
Applying the Sobolev embedding, for given $s\in[2,\infty)$, there exists $\gamma_s>0$ such that
    \begin{equation}
    \label{25}
      \|u\|_s\leq\gamma_s\|u\|,\;\;\forall\;u\in X_p.
    \end{equation}
    By (F1) and (F5), for any $\epsilon > 0$, there exists some constant $C_\epsilon>0$ such that
    \begin{equation}
    \label{26}
      |F(x,t)|\leq\epsilon t^2+C_\epsilon(e^{3\alpha_0 t^2/2}-1)|t|^3,\;\;\forall\;(x,t)\in\mathbb{R}^2\times\mathbb{R}.
    \end{equation}
    On the other hand, in view of Lemma \ref{M-T}, one has
    \begin{equation}
    \label{27}
      \int_{\mathbb{R}^2}\left(e^{3\alpha_0 u^2}-1\right)dx\leq C,\;\;\forall\;\|u\|\leq\sqrt{\frac{\pi}{\alpha_0}}.
    \end{equation}
    Let $\epsilon=\frac{1}{4\gamma_2^{2}}$, from \eqref{25}-\eqref{27}, there holds
    \begin{equation}
    \label{28}
      \int_{\mathbb{R}^2}F(x,u)dx\leq\frac{1}{4}\|u\|^2+C_3\|u\|^3,\;\;\forall\;\|u\|\leq\sqrt{\frac{\pi}{\alpha_0}}.
    \end{equation}
    Hence, it follows from \eqref{Phi} and \eqref{28} that if $\|u\|\leq\sqrt{\frac{\pi}{\alpha_0}}$,
    \begin{align}
    \label{29}
    \begin{split}
    \Phi(u)&=\frac{1}{2}\|u\|^2+\frac{1}{4p\pi}(I_1(u)-I_2(u))-\int_{\mathbb{R}^2}F(x,u)dx\\
    &\geq\frac{1}{4}\|u\|^2-C_3\|u\|^3-C_4\|u\|^{2p}.
    \end{split}
    \end{align}
    Therefore, there exists $\kappa_0>0$ and $0<\rho<\sqrt{\frac{\pi}{\alpha_0}}$ such that
    \begin{equation}
      \Phi(u)\geq\kappa_0,\;\;\forall\; u\in S:=\{u\in E_{k,p}:\|u\|=\rho\}.
    \end{equation}
    Since $\lim_{t\rightarrow\infty}\Phi(t\omega_{n_0})=-\infty$, we can choose $t^*>0$ such that $e=t^*\omega_{n_0}\in Y_2:=\{u\in E_{k,p}:\|u\|>\rho\}$ and $\Phi(e)<0$. Let $Y_1:=\{u\in E_{k,p}:\|u\|\leq\rho\}$, then in view of Lemma \ref{lem9}, one deduces that there exists $\tilde{c}\in[\kappa_0,\sup_{t\geq0}\Phi(t\omega_{n_0})]$ and a Cerami sequence $\{u_n\}\subset E_{k,p}$ satisfying \eqref{24}.
  \end{proof}

\begin{lemma}
\label{lem12}
Assume that (V0),(VF),(F1) and (F4)-(F5) hold. Then any sequence satisfying \eqref{24} is bounded w.r.t.~$\|\cdot\|$.
\end{lemma}
\begin{proof}
Without loss of generality, we assume that $\mu_1\in(2,2p]$ and $0 \leq \mu_2 \leq \frac{\mu_1}{2}-1$. Choosing $\lambda_0\in\left(\frac{1}{\mu_1},\frac{1}{2}-\frac{\mu_2}{\mu_1}\right)$, by \eqref{Phi},\eqref{diff} and Lemma \ref{lem9},there holds
  \begin{align}
  \label{31}
  \begin{split}
  \tilde{c}+o(1)& = \Phi(u_n)-\lambda_0\langle\Phi^\prime(u_n),u_n\rangle_{\langle X_p^\prime,X_p\rangle}\\
  & = \left(\frac{1}{2}-\lambda_0\right)\|u_n\|^2+\frac{1}{2\pi}\left(\frac{1}{2p}-\lambda_0\right)I_0(u_n)\\
  &\quad +\int_{\mathbb{R}^2}[\lambda_0f(x,u_n)u_n-F(x,u_n)]dx\\
  & \geq\left(\frac{1}{2}-\lambda_0\right)\|u_n\|^2-\frac{\mu_2}{\mu_1}\int_{\mathbb{R}^2}V(x)u_n^2dx+\left(\lambda_0-\frac{1}{\mu_1}\right)\int_{\mathbb{R}^2}f(x,u_n)u_ndx\\
  & \geq \left(\frac{1}{2}-\frac{\mu_2}{\mu_1}-\lambda_0\right)\|u_n\|^2 +\left(\lambda_0-\frac{1}{\mu_1}\right)\int_{\mathbb{R}^2}f(x,u_n)u_ndx.
  \end{split}
  \end{align}
  We derive that $\|u_n\|$ is bounded.
\end{proof}

\section{Proof of Theorem \ref{thm1}}
Now we are ready to prove Theorem \ref{thm1}.  We only consider the function space $E_{k,p}$, because the $T_{k,p}$ case is very similar.

Applying Lemmas \ref{lem11} and \ref{lem12}, there exists a sequence $\{u_n\}\subset E_{k,p}$ satisfying \eqref{24} and $\|u_n\|\leq C$ for some constant $C>0$. Moreover, by \eqref{31}, we have
\begin{equation}
\label{32}
  \int_{\mathbb{R}^2}f(x,u_n)u_ndx\leq C.
\end{equation}
Next, we prove the Theorem \ref{thm1} in three steps.

\medskip
\noindent\textbf{Step 1:} $\{u_n\}$ is bounded in $E_{k,p}$.

We claim first $\delta_0:=\limsup_{n\rightarrow\infty}\|u_n\|_p > 0$. Suppose the contrary $\delta_0 = 0$, then from the Gagliardo-Nirenberg inequality (see \cite[p.125]{Nirenberg1959}):
\begin{equation}
  \|u_n\|^{s}_{s}\leq C_s\|u_n\|_{p}^{\theta}\|\nabla u_n\|_{2}^{1-\theta},
\end{equation}
where $2\leq p<t<\infty,\theta=\frac{p}{t}$. Hence
$u_n\rightarrow0$ in $L^{\eta}(\mathbb{R}^2)$ for $\eta\in(2,+\infty)$. Given any $\varepsilon\in(0,M_0C_{10}/t_2)$, we choose $M_\varepsilon>M_0C_{10}/\varepsilon$, then it follows from (F2) and \eqref{32} that
\begin{align}
\label{34}
\begin{split}
  \int_{|u_n|\geq M_\varepsilon}F(x,u_n)dx&\leq M_0\int_{|u_n|\geq M_\varepsilon}|f(x,u_n)|dx\\
  &\leq\frac{M_0}{M_\varepsilon}\int_{|u_n|\geq M_\varepsilon}f(x,u_n)u_ndx<\varepsilon.
\end{split}
\end{align}
Applying (F5), one has
\begin{equation}
\label{35}
\int_{|u_n|\leq M_\varepsilon}F(x,u_n)dx\leq
  \begin{cases}
  C_\varepsilon\|u_n\|_{2}^{2}=o(1),&\text{$p=2$},\\
  C_\varepsilon\|u_n\|_{s+1}^{s+1}=o(1),&\text{$p>2$}\\
  \end{cases}
\end{equation}
and
\begin{equation}
\label{36}
  \int_{|u_n|\leq1}f(x,u_n)u_ndx\leq
  \begin{cases}
  C\|u_n\|_{2}^{2}=o(1),&\text{$p=2$},\\
  C\|u_n\|_{s+1}^{s+1}=o(1),&\text{$p>2$}.\\
  \end{cases}
\end{equation}
By the arbitrariness of $\varepsilon>0$, we deduce from (F2),\eqref{34} and \eqref{35} that
\begin{equation}
\label{37}
  \int_{\mathbb{R}^2}F(x,u_n)dx=o(1).
\end{equation}
Hence, by \eqref{10},\eqref{Phi},\eqref{24} and \eqref{37}, Lemmas \ref{lem10} and \ref{lem11}, there is $\bar{\varepsilon}>0$ such that
\begin{equation}
\label{38}
  \|u_n\|^2\leq2\tilde{c}+o(1):=\frac{4\pi}{\alpha_0}(1-3\bar{\varepsilon})+o(1).
\end{equation}
Now let $d\in(1,\frac{p}{p-1})$ satisfying
\begin{equation}
\label{39}
  \frac{(1+\bar{\varepsilon})(1-3\bar{\varepsilon})d}{1-\bar{\varepsilon}}<1.
\end{equation}
By (F1), there exists $C>0$ such that
\begin{equation}
\label{40}
  |f(x,t)|^d\leq C\left[e^{\alpha_0(1+\bar{\varepsilon})dt^2}-1\right],\;\;\forall\;x\in\mathbb{R}^2,|t|\geq1.
\end{equation}
It follows from \eqref{38}-\eqref{40} and Lemma \ref{M-T} that
\begin{align}
\label{41}
  \int_{|u_n|\geq1}|f(x,u_n)|^ddx&\leq C\int_{\mathbb{R}^2}\left[e^{\alpha_0(1+\bar{\varepsilon})du_n^2}-1\right]dx\nonumber\\
  &=C\int_{\mathbb{R}^2}\left[e^{\alpha_0(1+\bar{\varepsilon})d\|u_n\|^2(u_n/\|u_n\|)^2}-1\right]dx\leq C.
\end{align}
As $d^{\prime}= \frac{d}{d-1} >p$, using \eqref{41}, there holds
\begin{equation}
\label{42}
  \int_{|u_n|\geq1}f(x,u_n)u_ndx\leq\left[\int_{|u_n|\geq1}|f(x,u_n)|^qdx\right]^{1/d}\|u_n\|_{d^{\prime}}=o(1).
\end{equation}
Combining \eqref{10}-\eqref{diff},\eqref{24},\eqref{36} and \eqref{42}, we arrive at
\begin{align}
\begin{split}
  \tilde{c}+o(1)=&\;\Phi(u_n)-\frac{1}{2}\langle \Phi^{\prime}(u_n),u_n\rangle\\
  =&-\left(\frac{1}{4\pi}-\frac{1}{4p\pi}\right)I_1(u_n)+\left(\frac{1}{4\pi}-\frac{1}{4p\pi}\right)I_2(u_n)\\
  &+\int_{\mathbb{R}^2}\left[\frac{1}{2}f(x,u_n)u_n-F(x,u_n)\right]dx\\
  \leq&\; o(1).
\end{split}
\end{align}
This contradiction shows that $\delta_0>0$. Now from \eqref{10} and Lemmas \ref{lem6} and \ref{lem12}, one has
\begin{equation*}
  I_1(u_n)\leq I_2(u_n)\leq C,
\end{equation*}
which, together with Lemma \ref{lem8}, implies that $\|u_n\|_{\ast}$ is bounded, and so $\{u_n\}$ is bounded in $E_{k,p}$.

\medskip
\noindent\textbf{Step 2:} $\Phi'(\bar{u})=0$ in $E_{k,p}^{\prime}$.

We may assume, by \cite[Lemma 2.3]{Cao_Dai_Zhang} and passing to a subsequence again if necessary, that $u_n\rightharpoonup\bar{u}$ in $E_{k,p}$, $u_n\rightarrow\bar{u}$ a.e. on $\mathbb{R}^2$ and
\begin{equation*}
  u_n\rightarrow\bar{u}\;\;\mbox{in}\;\;L^{s}(\mathbb{R}^2),
\end{equation*}
where $s\in[2,\infty)$ if $p=2$ and $s\in(2,\infty)$ if $p>2$.
\noindent Let $M:=\sup_{n}\|\nabla u_n\|_2$. By \eqref{38}, one has $M^2<\frac{4\pi}{\alpha_0}$. Therefore, there is $\beta>p$ big enough such that $\frac{M^2\beta}{\beta-1}<\frac{4\pi}{\alpha_0}$. Without loss of generality, we assume that $s_0\in(1,2)$. Then it follows (F5) and Lemma \ref{M-T} that
\begin{align}
\label{45}
  &\;\int_{\mathbb{R}^2}|f(x,u_n)(u_n-\bar{u})|dx\nonumber\\
  \leq&\;\int_{\{|u_n|<1\}}|f(x,u_n)(u_n-\bar{u})|dx+\int_{\{|u_n|\geq1\}}|f(x,u_n)(u_n-\bar{u})|dx\nonumber\\
  \leq&\; C\|u_n\|_2\|u_n-\bar{u}\|_{\frac{2}{2-s_0}}+C\|u_n-\bar{u}\|_\beta&\nonumber\\
  =&\; o(1).
\end{align}
Similarly, one has
\begin{equation}
  \int_{\mathbb{R}^2}|f(x,\bar{u})(u_n-\bar{u})|dx=o(1).
\end{equation}
Furthermore, it follows from \eqref{9},\eqref{10} and H\"{o}lder inequality that
\begin{equation}
\label{47}
  A_2(|u_n|^p,|u_n|^{p-2}u_n(u_n-\bar{u}))=o(1),\;\;A_2(|\bar{u}|^p,|\bar{u}|^{p-2}\bar{u}(u_n-\bar{u}))=o(1).
\end{equation}
By \cite[Lemma 3.3]{Cao_Dai_Zhang}, we have
\begin{equation}
\label{47a}
  A_1(|u_n|^p,|u_n|^{p-2}\bar{u}(u_n-\bar{u}))=o(1),\;\;A_1(|\bar{u}|^p,|\bar{u}|^{p-2}\bar{u}(u_n-\bar{u}))=o(1).
\end{equation}
Combining \eqref{Phi},\eqref{diff},\eqref{24},\eqref{45}-\eqref{47a}, there holds
\begin{align}
\label{48}
\begin{split}
  o(1)=&\;\langle \Phi^{\prime}(u_n)-\Phi^{\prime}(\bar{u}),u_n-\bar{u}\rangle_{\langle X_p',X_p\rangle}\\
  =&\;\|u_n-\bar{u}\|^2+\frac{1}{2\pi}A_1(|u_n|^p,|u_n|^{p-2}(u_n-\bar{u})^2)\\
  &+\frac{1}{2\pi}A_1(|u_n|^p,|u_n|^{p-2}\bar{u}(u_n-\bar{u}))-\frac{1}{2\pi}A_1(|\bar{u}|^p,|\bar{u}|^{p-2}\bar{u}(u_n-\bar{u}))\\
  &+\frac{1}{2\pi}A_2(|\bar{u}|^p,|\bar{u}|^{p-2}\bar{u}(u_n-\bar{u}))-\frac{1}{2\pi}A_2(|u_n|^p,|u_n|^{p-2}u_n(u_n-\bar{u}))\\
  &+\int_{\mathbb{R}^2}f(x,\bar{u})(u_n-\bar{u})dx-\int_{\mathbb{R}^2}f(x,u_n)(u_n-\bar{u})dx+o(1)\\
  \geq&\;\|u_n-\bar{u}\|^2+o(1).
  \end{split}
\end{align}
By \eqref{10},\eqref{48} and Lemma \ref{M-T}, we have
\begin{equation}
\label{48a}
  A_2(|u_n|^p,|v_n|^{p})=o(1),\;\;\langle u_n,u_n-\bar{u}\rangle=o(1),\;\;\int_{\mathbb{R}^2}f(x,u_n)(u_n-\bar{u})dx=o(1),
\end{equation}
where $|v_n|^p:=|u_n|^{p-2}|u_n-\bar{u}|^2$ for every $n\in\mathbb{N}$. By \eqref{24} and \eqref{48a}, one has
\begin{align*}
  o(1)= \langle\Phi'(u_n),u_n-\bar{u}\rangle_{\langle X_p',X_p\rangle} & = \frac{1}{2\pi}A_1(|u_n|^p,|v_n|^p)-\frac{1}{2\pi}A_2(|u_n|^p,|v_n|^p)\\
  & \quad +\langle u_n,u_n-\bar{u}\rangle+\int_{\mathbb{R}^2}f(x,u_n)(u_n-\bar{u})dx\\
  & = \frac{1}{2\pi}A_1(|u_n|^p,|v_n|^p)+o(1)
\end{align*}
which, together with Lemma \ref{lem8}, implies
\begin{equation}
\label{49}
\lim_{n\rightarrow\infty}(\|v_n\|_p+\|v_n\|_{\ast})\rightarrow0.
\end{equation}
It is well known that for any $\epsilon > 0$,
\begin{equation}
\label{50}
\big||a+b|^\theta-|b|^\theta\big|\leq\epsilon|b|^\theta+C_{\epsilon,\theta}|a|^\theta,\;\forall\;a,b\in\mathbb{R}.
\end{equation}
From \eqref{49},\eqref{50} (take $\epsilon=\frac{1}{2}$ therein) and \cite[Lemma 3.3]{Cao_Dai_Zhang}, one has
\begin{align}
\begin{split}
  \|u_n-\bar{u}\|_{\ast}^p&=\int_{\mathbb{R}^2}\ln(1+|x|)(|u_n-\bar{u}|^{p-2}-|u_n|^{p-2})|u_n-\bar{u}|^2dx+o(1)\\
  &\leq\frac{1}{2}\|u_n-\bar{u}\|_{\ast}^p+C\int_{\mathbb{R}^2}\ln(1+|x|)|\bar{u}|^{p-2}|u_n-\bar{u}|^2dx+o(1)\\
  &=\frac{1}{2}\|u_n-\bar{u}\|_{\ast}^p+o(1).
  \end{split}
\end{align}
Combining with \eqref{48}, we have $u_n\rightarrow\bar{u}$ in $E_{k,p}$. Hence, $0<\tilde{c}=\lim_{n\rightarrow\infty}\Phi(u_n)=\Phi(\bar{u})$ and $\Phi^{\prime}(\bar{u})=0$ in $E_{k,p}$.

\medskip
\noindent\textbf{Step 3:} $\Phi'(\bar{u})=0$ in $X_{p}^{\prime}$.

Now we will conclude $\Phi^{\prime}(\bar{u})=0$, using the group action on the space $X_p$. The guideline of the proof holds for more general cases. Let $G\subset\mathcal{O}(2)$ be a finite group of transforms acting on $X_p$, where $\mathcal{O}(2)$ denotes the group of orthogonal transformations in $\mathbb{R}^2$. The action of $G$ on the space $X_p$ is a continuous map (see \cite[Definition 1.27]{Willem1996}):
$$G\times X_p\rightarrow X_p:[\tau,u]\rightarrow \tau(u) = u \circ \tau.$$
Assume that $\varphi\in C^1(X_p,\mathbb{R})$ is invariant by $G$, that is $\varphi(w\circ\tau) = \varphi(w)$ for any $\tau \in G$, $w \in X_p$. Let $u$ be a critical point of $\varphi$ in $X_{p,G}$, where
$$X_{p,G}:=\{u\in X_p:\tau u=u,\;\forall\;\tau\in G\}.$$
We claim that $\varphi^{\prime}(u)=0$ in $X_p$. In fact, given any $v\in X_p$, we define
$$\bar{v}=\frac{1}{\#(G)}\sum_{\tau\in G}\tau v,$$
where $\#(G)$ denotes the cardinal of $G$. Clearly, $\bar{v}\in X_{p,G}$. Therefore, one has
\begin{align*}
0 = \langle\varphi^{\prime}(u),\bar v\rangle = \frac{1}{\#(G)}\sum_{\tau\in G} \langle \varphi^{\prime}(u), \tau v\rangle & = \frac{1}{\#(G) }\sum_{\tau\in G} \langle \varphi'(u)\circ\tau^{-1}, v\rangle\\
&  = \frac{1}{\#(G) }\sum_{\tau\in G} \langle \varphi'(u\circ\tau^{-1})\circ\tau^{-1}, v\rangle\\
& = \langle\varphi^{\prime}(u),v \rangle.
\end{align*}
For the second line, we used the fact $u \in X_{p, G}$. So we have $\varphi^{\prime}(u)=0$ in $X_p$.

The two cases in Theorem \ref{thm1} are direct consequence of the above discussion. Indeed, let $G_1$ be the subgroup of $\mathcal{O}(2)$ generated by $z\mapsto ze^{2\pi i/k}$, then $X_{p, G_1} = E_{k, p}$. If $G_2$ is generated by $z\mapsto ze^{2\pi i/k}$ and $z\mapsto \bar z$, $X_{p, G_2} = T_{k, p}$. So $\Phi'(\bar{u}) = 0$ in $X_p'$.\hfill$\Box$

\section{Proof of Theorem \ref{thm3}}

To prove Theorem \ref{thm3}, we show several Lemmas. As in the previous section, we will only consider the $E_{k,p}$ case.
\begin{lemma}
\label{lem13}
Assume that (V0),(VF),(F1),(F4') and (F5) hold, the Nerahi manifold $\mathcal{N}_1$ satisfies
\begin{itemize}
\item[(1)] There exists a constant $c_\ast\in(0,m_1]$ and a sequence $\{u_n\}\subset E_{k,p}$ satisfying
  \begin{equation}
  \label{53}
    \Phi(u_n)\rightarrow c_\ast,\;\;\|\Phi^{\prime}(u_n)\|_{X_p^{\prime}}(1+\|u_n\|_{X_p})\rightarrow0,
  \end{equation}
  where $m_1:=\inf_{\mathcal{N}_1}\Phi(u)$.
\item[(2)] For any $u\in E_{k,p}\backslash\{0\}$, there exists a unique $t_u>0$ such that $t_uu\in\mathcal{N}_1$. Moreover, we have
$$m_1=\inf_{E_{k,p}\backslash\{0\}}\max_{t\geq0}\Phi(tu).$$
\end{itemize}
\end{lemma}
  \begin{proof}
Firstly, we prove that
  \begin{equation}
  \label{54}
    \Phi(u)\geq \Phi(tu)+\frac{1-t^{2p}}{2p}\langle \Phi^\prime(u),u\rangle+\frac{t^{2p}-pt^2+p-1}{2p}\|u\|^2,\;\;\forall\; u\in X_p,\;t\geq0.
  \end{equation}
    Indeed, by \eqref{Phi} and \eqref{diff}, one has
    \begin{align*}
      \Phi(u)-\Phi(tu) & = \frac{1-t^2}{2}\|u\|^2+\frac{1-t^{2p}}{4p\pi}I_0(u)+\int_{\mathbb{R}^2}[F(x,tu)-F(x,u)]dx\\
      & = \frac{1-t^{2p}}{2p}\langle \Phi^{\prime}(u),u\rangle+\frac{t^{2p}-pt^2+p-1}{2p}\|u\|^2\\
      &\quad +\int_{\mathbb{R}^2}\left[\frac{1-t^{2p}}{2p}f(x,u)u+F(x,tu)-F(x,u)\right]dx\\
      & = \frac{1-t^{2p}}{2p}\langle \Phi^{\prime}(u),u\rangle+\frac{t^{2p}-pt^2+p-1}{2p}\|u\|^2\\
      &\quad +\int_{\mathbb{R}^2}\int_{t}^{1}\left[\frac{f(x,u)- V(x)u}{|u|^{2p-1}}-\frac{f(x,su)-V(x)su}{|su|^{2p-1}}\right]s^{2p-1}|u|^{2p-1}udsdx\\
      & \geq \frac{1-t^{2p}}{2p}\langle \Phi^{\prime}(u),u\rangle+\frac{t^{2p}-pt^2+p-1}{2p}\|u\|^2.
    \end{align*}
    Secondly, let $g(t):=t^{2p}-pt^2+p-1$, there holds
    $$\min_{t\geq0}g(t)=g(1)=0.$$
    Hence, by \eqref{54}, one has
    \begin{equation}
    \label{54a}
    \Phi(u)=\max_{t\geq0}\Phi(tu),\;\;\forall\; u\in\mathcal{N}_1.
    \end{equation}
    Finally, choose $v_k\in\mathcal{N}_1$ such that
    $$m_1\leq\Phi(v_k)\leq m_1+\frac{1}{k},\;\;k\in\mathbb{N}.$$
    For any $v_k$, similarly to Lemma \ref{lem11}, we can obtain a Cerami sequence $\{u_{k,n}\}_n\subset E_{k,p}$ such that
    $$\Phi(u_{k,n})\rightarrow c_k,\;\;\|\Phi^\prime(u_{k,n})\|_{X_p^{\prime}}(1+\|u_n\|_{X_p})\rightarrow0,\;\;k\in\mathbb{N}$$
    with $c_k\in(0,\sup_{t\geq0}\Phi(tv_k)]$. By \eqref{54a} and the diagonal rule, we get a Cerami sequence $\{u_n\}\subset E_{k,p}$ such that
    \begin{equation}
    \Phi(u_n)\rightarrow c_\ast,\;\;\|\Phi^{\prime}(u_n)\|_{X_p^{\prime}}(1+\|u_n\|_{X_p})\rightarrow0
    \end{equation}
    with $c_\ast\in(0,m_1]$. Here we refer the readers to \cite[Lemma 3.2]{Chen_Shi_Tang} for more detail. So we can claim the point (1).

    \medskip
    Next, we consider the point (2). Let $u\in E_{k,p}\backslash\{0\}$ be fixed and $\zeta(t):=\Phi(tu)$ on $[0,\infty)$. By the definition \eqref{Phi},
    $$\zeta^{\prime}(t)=0\iff t^2\|u\|^2+\frac{t^{2p}}{2\pi}I_0(u)-\int_{\mathbb{R}^2}f(x,tu)tudx=0\iff tu\in\mathcal{N}_1.$$
    Using \eqref{29},(F1) and Lemma \ref{lem6}, one has $\zeta(0)=0$, $\zeta(t)>0$ for $t>0$ small and $\zeta(t)<0$ for $t$ large. Therefore $\max_{t\in(0,\infty)}\zeta(t)$ is achieved at some $t_u>0$ so that $\zeta^{\prime}(t_u)=0$ and $t_uu\in\mathcal{N}_1$. Now we claim that $t_u$ is unique. In fact, for any given $u\in E_{k,p}\backslash\{0\}$, let $t_1,t_2>0$ such that $\zeta^{\prime}(t_1)=\zeta^{\prime}(t_2)=0$. By \eqref{54}, taking $t=\frac{t_2}{t_1}$ and $t=\frac{t_1}{t_2}$ respectively, it implies
    $$\Phi(t_1u)\geq\Phi(t_2u)+\frac{t_1^2 g(t_2/t_1)}{2p}\|u\|^2, \quad \mbox{and}\quad \Phi(t_2u)\geq\Phi(t_1u)+ \frac{t_2^2 g(t_1/t_2)}{2p}\|u\|^2$$
    Therefore, we must have $t_1=t_2$, since $g(s) > 0$ for any $s > 0$, $s\ne 1$.
  \end{proof}

\begin{lemma}
\label{lem14}
Assume that (V0),(VF),(F1),(F4') and (F5) hold. Then any sequence satisfying \eqref{53} is bounded w.r.t. $\|\cdot\|$.
\end{lemma}
\begin{proof}
We only considers the case $p > 2$. For $p=2$, the readers can refer to \cite[Lemma 2.11]{Chen_Tang2020-2}. First, we claim
    \begin{equation}
    \label{55}
    \frac{1}{2p}f(x,t)t-F(x,t)\geq\frac{\mu(1-p)}{2p}V(x)t^2,\;\forall\; t\in\mathbb{R}.
    \end{equation}
    Indeed, by (F4'), there holds
    \begin{align*}
      F(x,t)-\frac{\mu}{2}V(x)t^2&=\int_0^t[f(x,\tau)-\mu V(x)\tau]d\tau\\
      &\leq\int_0^t\frac{f(x,t)-\mu V(x)t}{|t|^{2p-1}}|\tau|^{2p-2}\tau d\tau\\
      &=\frac{f(x,t)t-\mu V(x)t^2}{2p}.
    \end{align*}
    By \eqref{55}, one has
    \begin{align}
      c_\ast+o(1)&=\Phi(u_n)-\frac{1}{2p}\langle \Phi^{\prime}(u_n),u_n\rangle\notag\\
      &=\left(\frac{1}{2}-\frac{1}{2p}\right)\|u_n\|^2+\int_{\mathbb{R}^2}\left(\frac{1}{2p}f(x,u_n)u_n-F(x,u_n)\right)dx\notag\\
      &\geq\left(\frac{1}{2}-\frac{1}{2p}\right)\|u_n\|^2-\left(\frac{1}{2}-\frac{1}{2p}\right)\mu\int_{\mathbb{R}^2}V(x)u_n^2dx\notag\\
      &\geq\frac{(p-1)(1-\mu)}{2p}\|u_n\|^2.
    \end{align}
    Here we used also \eqref{Phi},\eqref{diff} and \eqref{53}.
  \end{proof}

\medskip
\noindent\textbf{Proof of Theorem \ref{thm3} completed.} Applying Lemmas \ref{lem13} and \ref{lem14}, we deduce that there exists a sequence $\{u_n\}\subset E_{k,p}$ satisfying \eqref{53} and $\|u_n\|\leq C < \infty$. Now we prove
\begin{equation}
\label{57}
  \int_{\mathbb{R}^2}f(x,u_n)u_ndx\leq C.
\end{equation}
Indeed, let $p \ge 2$, by \eqref{Phi}-\eqref{diff} and \eqref{53}, there holds
\begin{align*}
  c_\ast+o(1)=&\;\Phi(u_n)-\frac{p+\mu(1-p)}{2p}\langle\Phi^\prime(u_n),u_n\rangle\\
  \geq&\frac{(p-1)\mu}{2p}\|u_n\|^2-\frac{(p-1)\mu}{2p}\int_{\mathbb{R}^2}V(x)u_n^2dx+\frac{(1-p)(1-\mu)}{4p\pi}I_0(u_n)\\
  &+\frac{(p-1)(1-\mu)}{2p}\int_{\mathbb{R}^2}f(x,u_n)u_ndx\\
  \geq&\frac{(p-1)(1-\mu)}{2p}\int_{\mathbb{R}^2}f(x,u_n)u_ndx,
\end{align*}
hence \eqref{57} holds true.
Replacing \eqref{32} by \eqref{57} in the the proof of Theorem \ref{thm1} and proceeding as before, one can conclude the existence of $\bar u \in E_{k, p}\backslash \{0\}$ with $\Phi'(\bar{u})=0$ in $X_p'$. Obviously, $\bar{u}\in\mathcal{N}_1$, which together with $c_{\ast}\leq m_1$ (see Lemma \ref{lem13}) implies that $\Phi(\bar{u})=m_1$. \hfill$\Box$

\section{Further discussions}
\label{further}
In this section, we discuss several open problems above the 2-D Schr\"{o}dinger-Poisson system
\begin{equation*}
\begin{cases}
  -\triangle{u}+V(x)u+ K(x)\phi{(x)}|u|^{p-2}u=f(x,u),&\text{in $\mathbb{R}^{2}$},\\
  \triangle{\phi}=K(x)|u|^{p},&\text{in $\mathbb{R}^{2}$},
\end{cases}
\leqno{(O)}
\end{equation*}
where $V,K\in C(\mathbb{R}^2\backslash\{0\},\mathbb{R}_+),f\in C(\mathbb{R}^2\backslash\{0\}\times\mathbb{R},\mathbb{R})$ and $p>1$. We are interested in the existence of solutions using the variational method.


\medskip
In this work, we considered mainly the case $K(x)\equiv1$, $p \ge 2$ and $f(x,t)$ is of critical growth in $t=\pm\infty$ in Theorems \ref{thm1} and \ref{thm3}, and made some progress comparing to \cite[Theorem 1.4]{Chen_Tang2020-2}, It seems that there are still some rooms for improvement. One idea is to construct better test functions such that Lemma \ref{lem10} holds. Another idea is to use the concentration-compactness principle. For that, we could use the idea coming from \cite{Liu_Radulescu_Tang_Zhang}. In other words, we consider the equation
\begin{equation*}
  -\triangle{u}+V(x)u+ \left[G_\alpha(|\cdot|)\ast\left(K(x)|u|^{p}\right)\right]|u|^{p-2}u=f(x,u)\quad \mbox{in  } \mathbb{R}^2,
\leqno{(O_\alpha)}
\end{equation*}
where $G_\alpha:=\frac{|x|^{-\alpha}-1}{\alpha}$ and $\alpha\in(0,1)$. By the fact $\lim_{\alpha\rightarrow0^+}\frac{|x|^{-\alpha}-1}{\alpha}=-\ln|x|$ for all $x\in\mathbb{R}^2\backslash\{0\}$, we may obtain the solution of the equation $(O)$.

\medskip
As far as we know, there is no existence result for $p \in (1, 2)$ even with $K(x)\equiv1$. Perhaps the crucial point is to prove the compactness of Cerami sequence, see the inequality \eqref{48} for $p \ge 2$.

\medskip
For $K(x)\neq1$, In \cite{Albuquerque_Carvalho_Figueiredo_Medeiros}, Albuquerque, Carvalho, Figueiredo and Medeiros proved recently the existence of ground state of the equation $(O)$ when $V,K$ and $f$ are radial. Therefore, it is natural to explore the general case with non-radial $V,K$ and $f$.

As mentioned in the Introduction, our approach works also for the subcritical case, we omit the proof here since it is very similar to the critical case.
\begin{theorem}
\label{R15}
  Let $p\geq2$, $V$ and $f$ satisfy (V0),(VF),(F1$^\prime$),(F5) and the following condition:
\begin{itemize}
\item[(F4')] $f(x,t)t>0$ for all $(x,t)\in\mathbb{R}^2\times(\mathbb{R}\backslash\{0\})$ and there exists $\mu\in(2,\infty),t_1\in(0,\infty)$ such that
$$f(x,t)t\geq\mu F(x,t),\;\;\forall\;x\in\mathbb{R}^2,|t|\geq t_1;$$
Furthermore, if $p > 2$, we assume that
  \begin{equation*}
    M_{t_1}<\left(\frac{1}{2}-\frac{1}{\mu}\right)\gamma^2
  \end{equation*}
  where $t_1 > 0$,
  \begin{equation*}
    M_{t_1}= \sup_{(x,t)\in\mathbb{R}^2\times[-t_1,t_1]\backslash\{0\}}\frac{F(x,t)}{t^2} \quad \mbox{and} \quad \gamma =\inf_{u\in X_p}{\frac{\|u\|}{\|u\|_{H^1(\mathbb{R}^2)}}}>0.
  \end{equation*}
\end{itemize}
Then $(P)$ has a nontrivial solution $\bar{u}$. Moreover, $\bar{u}\in E_{k,p}$ if $V\in\mathcal{V}_{k,1}$ and
$f\in\mathcal{F}_{k,1}$; $\bar{u}\in T_{k,p}$ if $V\in\mathcal{V}_{k,2}$ and $f\in\mathcal{F}_{k,2}$.
\end{theorem}

Here $\|\cdot\|$ is defined by \eqref{3}. The fact $\gamma > 0$ was proved by \cite[Lemma 2.1]{Chen_Tang2020-1}. If we suppose that $f(x, t) = 0$ for $t \leq 0$, and $f(x, t)$ is subcritical at $t\to \infty$, our proof will provide us a positive solution, roughly speaking, we can replace $\bar{u}$ with $|\bar{u}|$ in variational approach, the similar conclusion is valid also for the critical case, namely Theorem \ref{thm1}.

\medskip
Recently, Liu, R\u{a}dulescu, Tang and Zhang \cite{Liu_Radulescu_Tang_Zhang} considered the equation $(O)$ with $V=K \equiv 1$ and $f(x, u) = f(u)$.
They proved that the equation $(O)$ has a positive solution by assuming that:
\begin{itemize}
\item[(SF1)] $f\in C^1(\mathbb{R},\mathbb{R}^+)$, $f(s)\equiv0$ for $s\leq0$ and $f(s)>0$ for $s>0$;
\item[(SF2)] For every $\theta>0$, there exists $C_\theta>0$ such that $f(s)\leq C_\theta\min\{1,s\}e^{\theta s^2}$, $\forall\; s>0.$
\item[(SF3)] $\frac{f(s)}{s^{3}}$ is non-decreasing for $s>0$.
\end{itemize}

Clearly, the condition $(SF2)$ implies that $f$ is subcritical at $t = \infty$, and $(SF3)$ yields the $(AR)$ condition and $f(s) = O(s^3)$ at the origin. We can remark that Theorem \ref{R15} generalized the above result and also \cite[Theorem 1.2]{Chen_Tang2020-2}, \cite[Theorem 1.1]{Cao_Dai_Zhang}.


\section{Appendix}
In this section, we will give the proof of Lemma \ref{lem5}. In fact, we just need to prove $\Psi\in C^{1}(X_p,\mathbb{R})$, and
    the readers can refer to \cite[Lemma 2.3]{Cao_Dai_Zhang} for the rest. First, Given any $u,v\in X_p$, for almost every $x\in\mathbb{R}^2$
    $$\lim_{t\rightarrow0}\frac{F(x,u(x)+tv(x))-F(x,u(x))}{t}=f(x,u(x))v(x).$$
    On the other hand, we can choose large enough number $t_1>0$ such that
    $$|f(x,t)|\leq e^{(\alpha_0+1)t^2}-1,\;\forall\;|t|\geq t_1.$$
    By (F1),(F5) and Lemma \ref{M-T}, one has
    \begin{align*}
      \int_{\mathbb{R}^2}|f(x,u)|^2dx&=\int_{\{|u|\leq t_1\}}|f(x,u)|^2dx+\int_{\{|u|\geq t_1\}}|f(x,u)|^2dx\\
      &\leq C\|u\|_2^2+\int_{\mathbb{R}^2}\left(e^{(\alpha_0+1)u^2}-1\right)dx\\
      &\leq C.
    \end{align*}
    Then, $f(x,u)\in L^2(\mathbb{R}^2)$, $\forall\;u\in X_p,$ This will imply that the Gateaux derivative $\Psi_g^{\prime}(u)$ exists and $\Psi_g^{\prime}(u)\in X_p^{\prime}$.

    Let now $\{u_n\}\subset X_p$, $\|u_n-\bar{u}\|_{X_p}\rightarrow0$. Hence $u_n\rightarrow\bar{u}$ in $H^1(\mathbb{R}^2)$.
    Let us prove $\Psi'(\bar u) =\lim_{n\to \infty}\Psi'(u_n)$. It suffices to prove
    $$\lim_{n\rightarrow\infty} \sup_{\|v\|_{X_p} = 1} \left|\int_{\mathbb{R}^2}[f(x,u_n)-f(x,\bar{u})]vdx\right| =0.$$
    Defining $M:=\sup_{n}\|\nabla u_n\|_2$, then we will prove this Lemma in two cases.

\medskip
    \noindent\textbf{Special case: $M\leq\sqrt{\frac{\pi}{2\alpha_0}}$}.

    For any given $\varepsilon\in(0,1)$, we choose $R_{1}^{\varepsilon}=\exp{(1/\varepsilon^p)}-1>0$, then
     $$\|v\|_{L^p(B_R^c)}\leq\|v\|_{L^p(B_{R_{1}^{\varepsilon}}^c)}=\varepsilon[\ln(1+R_{1}^{\varepsilon})]^{1/p}\|v\|_{L^p(B_{R_{1}^{\varepsilon}}^c)}\leq\varepsilon\|v\|_{X_p},\;\forall\;R\geq R_{1}^{\varepsilon}.$$
    Furthermore, there is a $R_{2}^{\varepsilon}>0$ such that
    $$\|\bar{u}\|_{L^2(B_R^c)}\leq\|\bar{u}\|_{L^2(B_{R_{2}^{\varepsilon}}^c)}\leq\varepsilon,\;\forall\;R\geq R_{2}^{\varepsilon}.$$
    Let $R_\varepsilon=R_{1}^{\varepsilon}+R_{2}^{\varepsilon}$, by Lemma \ref{M-T} and $\alpha_0\le \frac{\pi}{2M^2}$, it is easy to verify that there is a $C_0>0$ such that $\|u\|_4\leq C_0\|u\|_{X_p}$ for all $u\in X_p$ and
    $$\int_{B_{R_\varepsilon}}|f(x,u_n)|^2dx + \int_{B_{R_\varepsilon}}|f(x,u_n)|^2|u_n|dx\leq C,\;\;\forall\;n.$$
    Now we claim that
    \begin{equation*}
    \limsup_{n\rightarrow\infty}\left\{\|f(x,u_n)-f(x,\bar{u})\|_{2} + \|u_n-\bar{u}\|_2\right\}\leq C\varepsilon,
    \eqno{(\ast\ast)}
    \end{equation*}
    where $C$ is independent of $n$.
    The proof of $(\ast\ast)$ comes from \cite[Lemma 2.1]{Figueriredo_Miyagaki_Ruf}. As $L^2(B_{R_\varepsilon})$ is a Hilbert space, we need only to prove
    $$\limsup_{n\rightarrow\infty}\int_{B_{R_\varepsilon}}(|f(x,u_n)|^2-|f(x,\bar{u})|^2)dx\leq C\varepsilon$$
  for the first part of $(**)$ and $C$ is independent of $n$.  Let $M^{\prime}$ be large enough such that
  $$\int_{\{|u_n|\geq M^{\prime}\}\cap B_{R_\varepsilon}}|f(x,u_n)|^2dx=\int_{\{|u_n|\geq M^{\prime}\}\cap B_{R_\varepsilon}}\frac{|f(x,u_n)|^2|u_n|}{|u_n|}dx\leq\frac{C_0}{M^{\prime}}\leq\varepsilon.$$
  By the dominated convergence and Fatou Lemma, one has
  \begin{align*}
  &\left|\int_{B_{R_\varepsilon}}(|f(x,u_n)|^2-|f(x,\bar{u})|^2)\right|\\
  &\leq\int_{\{|u_n|\geq M^{\prime}\}\cap B_{R_\varepsilon}}\frac{|f(x,u_n)|^2|u_n|}{M^{\prime}}dx+\int_{\{|u|\geq M^{\prime}\}\cap B_{R_\varepsilon}}\frac{|f(x,u_n)|^2|\bar{u}|}{M^{\prime}}dx\\
  &+\int_{B_{R_\varepsilon}}h_n(x)dx\\
  &=2\varepsilon+o_n(1),
  \end{align*}
  where $h_n(x):=\left||f(x,u_n(x))|^2\chi_{\{|u_n|<M^{\prime}\}\cap B_{R_\varepsilon}}-|f(x,\bar{u}(x))|^2\chi_{\{|u_n|<M^{\prime}\}\cap B_{R_\varepsilon}}\right|$ and we use the fact
  \begin{equation*}
    |h_n(x)|\leq
    \begin{cases}
    |f(x,\bar{u}(x))|^2,&\text{$|u_n|\geq M^{\prime}$},\\
    \sup\{|f(x,t)|:x\in\overline{B_{R_\varepsilon}},|t|<M^{\prime}\}+|f(x,\bar{u}(x))|^2,&\text{$|u_n|<M^{\prime}$}.
    \end{cases}
  \end{equation*}
  Therefore, we get $(\ast\ast)$. By $(\ast\ast)$, for large $n$, one has
    \begin{align*}
      &\left|\int_{\mathbb{R}^2}[f(x,u_n)-f(x,\bar{u})]vdx\right|\\
      \leq&\int_{B_{R_\varepsilon}}|f(x,u_n)-f(x,\bar{u})||v|dx+\int_{B_{R_\varepsilon^c}}|f(x,u_n)-f(x,\bar{u})||v|dx\\
      \leq&\frac{1}{\gamma}\varepsilon\|v\|_{X_p}+\frac{C}{\gamma}(\|u_n-\bar{u}\|_2+2\|\bar{u}\|_{L^2(B_{R_\varepsilon^c})})\|v\|_{X_p}\\
      &+2^{2/p}C\left(\int_{\mathbb{R}^2}\left[\exp\left(p'\alpha u_n^2\right)-1+\exp\left(p'\alpha \bar{u}^2\right)-1\right]dx\right)^{1/p'}\varepsilon\|v\|_{X_p}\\
      \leq&C\varepsilon\|v\|_{X_p},
    \end{align*}
    where $\frac{1}{p}+\frac{1}{p'}=1$ and $\gamma:=\inf_{u\in X}{\frac{\|u\|}{\|u\|_{H^1(\mathbb{R}^2)}}}>0$ (see \cite[Lemma 2.1]{Chen_Tang2020-1}).

\medskip\noindent    \textbf{General case: $M > 0$}.

    We can choose big enough bounded domain $B_{R}(0)$ and its bounded open coverage $\{\Omega_\ell\}_{\ell \le N_c}$ which has a partition of unity $w_\ell \;(1 \le \ell \le N_c)$ such that\\
    $$B_{R}(0)\subset\bigcup_{1 \le \ell \le N_c}\Omega_\ell,\;\sum_{\ell=1}^{N_c}w_\ell (x)=1,\;\forall\;x\in B_{R}(0);$$
    $$w_\ell \in C_c^1(\Omega_\ell),\;|\nabla w_\ell|\leq C,\;\forall \;\ell$$
    $$\int_{B_{R}^c(0)}|\nabla\bar{u}|^2dx\leq\frac{\pi}{4\alpha_0},\;\int_{B_{R}^c(0)}|\nabla u_n|^2dx\leq\frac{\pi}{4\alpha_0},\;\;\forall\;n;$$
    $$\int_{\Omega_\ell}|\bar{u}|^2dx\leq\frac{\pi}{8\alpha_0},\;\int_{\Omega_\ell}|u_n|^2dx\leq\frac{\pi}{8\alpha_0},\;\;\forall\;n, \ell;$$
    $$\int_{\Omega_\ell}|\nabla\bar{u}|^2dx\leq\frac{\pi}{8\alpha_0},\;\int_{\Omega_\ell}|\nabla u_n|^2dx\leq\frac{\pi}{8\alpha_0},\;\;\forall\;n, \ell.$$
    Repeating now the proof of the special case,  we can claim
    $$\int_{B_{R}^c(0)}\big|[f(x,u_n)-f(x,\bar{u})]v\big|dx\leq C\varepsilon\|v\|_X$$
    and
    $$\int_{\Omega_\ell}\left|[f(x,u_n)-f(x,\bar{u})]v\right|dx\leq C\varepsilon\|v\|_X,\;\;\forall\;l.$$
    Therefore, one has
    \begin{align*}
      &\left|\int_{\mathbb{R}^2}[f(x,u_n)-f(x,\bar{u})]vdx\right|\\
      \leq&\int_{B_{R}^c(0)}\left|(f(x,u_n)-f(x,\bar{u}))v\right|dx+\sum_{l=1}^{N_c}\int_{\Omega_i}\left|(f(x,u_n)-f(x,\bar{u}))v\right|dx\\
      \leq&(N_c+1)C\varepsilon\|v\|_X.
    \end{align*}
The continuity of Gateaux derivatives means that the functional is continuous differentiable. \hfill $\Box$

\bigskip
\noindent\textbf{Acknowledgements.}

The author would like to express my sincere gratitude to his tutors, Professors Dong Ye and Feng Zhou for their valuable guidance which provided insights that helped him to improve the paper. On the other hand, I would like to thank my doctoral classmates at ECNU, Houzhi Tang, Yuhao Yan, Weiming Zhang and Shitao Liu.


\bibliographystyle{spmpsci}      

\end{document}